\documentclass[final]{svjour3}
\usepackage{graphicx}
\usepackage{amssymb, amsmath}

\smartqed
\newcommand{\ang}[1]{\langle #1\rangle}
\newcommand{\norm}[1]{\left|\hspace{-1pt}\left\| #1\right\|\hspace{-1pt}\right|}
\newcommand{\bm}[1]{\mbox{{\boldmath$#1$}}} 

\newcommand{\T}{\mathcal{T}}
\newcommand{\dK}{\partial K}
\newcommand{\dKN}{\partial K\setminus\Gamma_N}

\def\rev#1{#1}
\begin{document}
  \title{Hybridized discontinuous Galerkin method for convection-diffusion problems}
  \titlerunning{HDG method for convection-diffusion problems}  
\author{Issei Oikawa} 
\institute{I. Oikawa 
          \at Organization for University Research Initiatives, Waseda University. \\
          }
\date{}
\journalname{}
\maketitle
  
\begin{abstract}
In this paper, we propose a new hybridized discontinuous Galerkin (DG) method for
the convection-diffusion problems with mixed boundary conditions.
\rev{A feature of the proposed method, is that it can greatly reduce the number of globally-coupled 
degrees of freedom}, 
compared with the classical DG methods.
The coercivity of a convective part is achieved by adding an upwinding term.  
We give error estimates of optimal order in the piecewise $H^1$-norm \rev{for general convection-diffusion problems}.
\rev{Furthermore, we prove that the approximate solution given by our scheme
is close to the solution of the purely convective problem when the viscosity coefficient is small.}
Several numerical results are presented to verify the validity  of our method.
\end{abstract}
 

\section{Introduction} 
The discontinuous Galerkin(DG) method\cite{ABCM02,BS08} is now widely applied to various problems in science and engineering because of its flexibility for the choices of approximate functions and element shapes. 
An issue of the DG method is, however, the size and band-widths of the resulting matrices could be much larger than those of the standard finite element method, 
 since the DG method is formulated in terms of the usual nodal values defined in each elements together with those corresponding to inter-element discontinuities. 
In order to surmount this difficulty, it is worth-while trying to extend the idea of the DG method by combining with the hybrid  displacement method 
 (see, for example, \cite{PW05,Tong70,KA72a,KA72b}). 
Thus, we introduce new unknown functions on inter-element edges. 
\rev{We can then obtain a formulation which results in a global system of 
equations involving only the inter-element unknowns. 
Consequently, the size of the system is smaller with respect to those of the classical DG methods.}
Recently, in \cite{KIO09,OK10,Oikawa10}, the author and his colleagues proposed and analyzed a new class of DG methods,
 a {\em hybridized} DG method, that is based on the hybrid displacement approach by stabilizing their old method \cite{KA72a,KA72b}.  
In \cite{KIO09}, we examined our idea by using a linear elasticity problem as a model problem and offered several numerical examples to confirm the validity of our formulation. 
After that, we carried out theoretical analysis by using the Poisson equation as a model problem. 
In \cite{Oikawa10}, the stability and convergence of symmetric and nonsymmetric interior penalty methods of hybrid type were studied. 
The usefulness of the lifting operator in order to ensure a better stability was also studied in \cite{Oikawa10}. 

For second-order elliptic problems,
Cockburn, Dong and Guzm{\'a}n provided  
the first analysis of  hybridization of the DG method  in \cite{CDG08}.
Furthermore, Cockburn and his colleagues are actively contributing to the {\em hybridizable} DG method
\cite{CGSS09,CGL09,CGS10}. They also developed hybridizable schemes for the Stokes problems \cite{CG09,CGCPS11,NPC10,CC12a,CC12b}  
 and the incompressible Navier-Stokes equations \cite{NPC11}.

For convection-diffusion problems, Cockburn et al.\cite{NPC09,CDGRS09} proposed  hybridized
 schemes in terms of numerical fluxes.
The stability of their methods is achieved by choosing the stabilization parameters according to the convection. 
\rev{They reported several numerical results exploring the
convergence properties of their schemes which were later theoretically proven in \cite{CCToAppear}.} 
 In \cite{ES10}, Egger and Sch{\"o}berl proposed a {\em hybridized mixed} method 
 stemmed from the original DG method (see, for example \cite{RH73,Richter92}). %
 Labeur and Wells proposed an upwind numerical flux and provided numerical results in \cite{LW07}.
  In \cite{Wells11}, the error analysis of the scheme proposed in \cite{LW07} for convection-diffusion equations was shown.
  The scheme we are going to propose is essentially the same as their one.
Our hybridized scheme is constructed to satisfy the coercivity on a convective part, while  the other schemes
 were obtained by introducing an upwind numerical flux.
 \rev{The author learned about Wells \cite{Wells11} after the completion of the present study. 
 Actually, the present work was presented firstly at \cite{Oikawa10p1,Oikawa10p2} in 2010. 
Wells \cite{Wells11} studied a kind of the hybridizable discontinuous Galerkin method for convection-diffusion equations.
 However, we provide error estimates for the general convection-diffusion cases, whereas only the purely convective and the purely diffusive ones are considered in \cite{Wells11}.
Moreover, our formulation admits arbitrary shapes of elements
 and the result reported in Section 5 is an actually new investigation. }

The purpose of this paper is to propose a  hybridized DG method for the stationary convection-diffusion problems, and to verify the stability of our scheme theoretically and numerically in the convection-dominated cases, that is, when the diffusive coefficient is very small. 

\smallskip 

Now let us formulate the continuous problem to be considered.
Let $\Omega$ be a bounded  polygonal or polyhedral domain in $\mathbb R^d$ $(d=2,3)$.
We consider the convection-diffusion problems with mixed boundary conditions:
\begin{subequations} \label{cdr}
\begin{eqnarray} 
    -\varepsilon \Delta u + \bm{b} \cdot \nabla u + cu &=& f  \textrm{ in } \Omega, \label{cdr1a} \\
    \varepsilon \nabla u \cdot \bm n &=& g_N  \textrm{ on } \Gamma_N, \label{cdr1b} \\
       u &=& 0  \textrm{ on } \Gamma_D, \label{cdr1c} 
\end{eqnarray}
\end{subequations}
where $\varepsilon>0$ is the diffusion coefficient; $f \in L^2(\Omega)$, $\bm b \in W^{1,\infty}(\Omega)^d$, $c \in L^\infty(\Omega)$, and $g_N \in H^{3/2}(\Omega)$
 are given functions. We assume $\overline{\Gamma_D \cup\Gamma_N} = \partial \Omega$, $\Gamma_D \cap \Gamma_N = \emptyset$,
 and that the inflow boundary is included in $\Gamma_D$, i.e.,
\[  
   \Gamma_- := \{ x \in \partial\Omega: \bm b(x) \cdot \bm n(x) < 0\} \subset \Gamma_D,
\]
where $\bm n$ is the outward unit normal vector to $\partial \Omega$.
Moreover, we assume that there exists a non-negative constant $\rho_0$ such that
\begin{eqnarray} \label{rho} 
   \rho(x) := c(x) -\frac 1 2 \mathrm{div} \bm b(x) \ge \rho_0 \ge 0, \quad \forall x \in \Omega.
\end{eqnarray}
Under these assumptions, the existence and uniqueness of a weak solution $u \in H^1(\Omega)$ follows from  the Lax-Milgram theorem.
We shall pose further regularity on $u$ in the error analysis.

This paper is organized as follows.
In Section 2, we introduce  finite element spaces  to describe our method, and norms and projections  to use in our error analysis. 
Section 3 is devoted to the formulation of our proposed hybridized DG method, and the mathematical analysis is given in Section 4. 
We explain why our proposed  method is stable even when $\varepsilon$ is close to  0 in Section 5. 
In Section 6, we report several results of numerical computations.
Finally, we conclude this paper in Section 7.

\section{Preliminaries}
\subsection{Notation}
\paragraph{Function spaces and norms}
Let $\mathcal T_h = \{K_i\}_i$  be a triangulation of $\Omega$ in the sense of \cite{OK10}.
Thus, each $K \in \mathcal T_h$ is a {\em star-shaped} $m$-polyhedral domain, where $m$ denotes an integer
$m \ge d+1$. The boundary $\partial K$ of $K \in \mathcal T_h$ is composed of $m$-faces.
We assume that $m$ is bounded from above independently a family of triangulations $\{\T_h\}_h$,
and $\partial K$ does not intersect with itself. 
We set $h = \max_{K \in \T_h}  h_K$, where $h_K$ denotes the diameter of $K$.
In this paper, we assume that  $\{\T_h\}_h$ is quasi-uniform.
The {\em skeleton} of $\mathcal T_h$ is 
defined by
\[
   \Gamma_h := \bigcup_{K\in\mathcal T_h} \partial K\setminus \Gamma_N.
\]
We introduce the {\em broken} Sobolev space over $\mathcal T_h$ defined by
\[
 H^k(\mathcal T_h) = \{ v \in L^2(\Omega) : v|_K \in H^k(K) \}
\]
and an $L^2$-space on $\Gamma_h$ defined by
\begin{eqnarray*}
 L^2_D(\Gamma_h) &= \{ \hat v \in L^2(\Gamma_h) : \hat v|_{\Gamma_D} = 0\}.
\end{eqnarray*}
Then, we set $\bm V = H^2(\mathcal T_h) \times L^2_D(\Gamma_h)$.
Throughout this paper, we denote an element in $\bm V$ by $\bm v = \{v,\hat v\}$. 
The inner products are defined as follows
\[
   (u,v)_K = \int_K uvdx, \qquad \ang{\hat u,\hat v}_e = \int_e \hat u\hat v ds,
\]
for $u$, $v$ $\in L^2(K)$ and $\hat u, \hat v \in L^2(e)$, where $K$ is an element of $\mathcal T_h$ and $e$ is an edge $e$ of $K$.
Let $\|\cdot\|_m$ and $|\cdot|_m$ be the usual Sobolev norms and seminorms in the sense of  \cite{AF03}, where $m$ is a positive integer.
We introduce  auxiliary seminorms:
\begin{eqnarray*}
   &&|v|_{m,h}^2 := \sum_{K\in\mathcal T_h} h_K^{2(m-1)}|v|_{m,K}^2 
    \quad \textrm{for} \  v \in H^m(\mathcal T_h), \\
   &&|\bm v|_{j,h}^2 := \sum_{K\in\mathcal T_h} \sum_{e \subset \partial K\setminus\Gamma_N}\left\| \sqrt{\frac{\eta_e}{h_e}} (\hat v -  v)\right\|^2_{0,e} 
    \quad \textrm{for} \ \bm v=\{v,\hat v\} \in \bm V,
\end{eqnarray*}
where $h_K$ is the diameter of $K$, $h_e$ is the length of $e$, and $\eta_e$ is a penalty parameter.
For error analysis, we define the {\em HDG}-norm as follows:  
\begin{eqnarray*}
 &&\norm{\bm  v}^2 := \norm{\bm v}^2_d + \norm{\bm v}^2_{rc}, \\
 && \norm{\bm v}^2_* := \norm{\bm v}^2_d + \norm{\bm v}^2_{rc,*},
\end{eqnarray*}
where 
\begin{eqnarray*}
 &&\norm{\bm v}_d^2 := \varepsilon  \left(|v|^2_{1,h} +  |v|^2_{2,h}+ |\bm v|_{j,h}^2 \right), \\
 &&\norm{\bm v}^2_{rc} :=  \sum_{K\in\mathcal T_h}
 \left( \| |\bm b \cdot \bm n|^{1/2}  (\hat v -v)\|^2_{0,\partial K\setminus\Gamma_N}\right) + \rho_0 \|v\|^2_{0,\Omega}, \\
 && \norm{\bm v}_{rc,*}^2 := \norm{\bm v}^2_{rc} + \|v\|_{0,\Omega}^2
 + \sum_{K\in\T_h}\|v\|_{0,\dK}^2. 
\end{eqnarray*}
Here $\bm v = \{v,\hat v\} \in \bm V$, $\bm n$ denotes the unit outward normal vector to $\partial K$, and $\rho_0$ is the positive constant defined in \eqref{rho}.
\paragraph{Finite element spaces and projections}
 Let $U_h$ and $\hat U_h$  be finite dimensional subspaces of $H^2(\mathcal T_h)$ and $L^2_D(\Gamma_h)$,
respectively. Then we set $\bm V_h := U_h \times \hat U_h$, which is included in $\bm V$. 
Let $P_h$ denote the $L^2$-projection from $H^2(\mathcal T_h)$ onto $U_h$, and
 let  $\hat P_h$ denote the $L^2$-projection from $L^2_D(\Gamma_h)$ onto $\hat U_h$.
Define $\bm P_h : \bm V \rightarrow \bm V_h$ by  $\bm P_h \bm v := \{P_h v, \hat P_h\hat v\}$
 and introduce the $L^2$-projection $\bm P_h^0  : W^{1,\infty}(\Omega)^d$ $\rightarrow \mathcal P^0(\mathcal T_h)^d$,
 where $\mathcal P^0(\mathcal T_h)$ is piecewise constant functions. 
In this paper, we assume that:
\begin{description}
\item{(H1)}
   $\nabla v_h \in [U_h]^d  \quad \forall v_h \in U_h.$
\item{(H2)}
   $\mathcal P^0(\mathcal T_h) \subset U_h.$
\item{(H3)}
   ({\em \rev{Approximation} properties})
   There exist positive constants $C$ independent of $\epsilon$ and $h$ such that, for all $v \in H^{k+1}(K)$,
   \begin{eqnarray*}
      &&| v - P_h v |_{i,K} \le C h^{k+1-i} |v|_{k+1,K}  \quad (i = 0,1), \\
      &&\| v - \hat P_h (v|_e) \|_{0,e} \le Ch^{k+1/2} |v|_{k+1,K},
   \end{eqnarray*}
   where $K \in \T_h$ and $e$ is an edge of $K$.
\end{description}
\rev{For example, we can take $U_h$ and $\hat U_h$ to be piecewise polynomials.   
 Although approximate functions on $\Gamma_h$ are allowed to be discontinuous at each vertex,
 we can use continuous functions for $\hat u_h$ to reduce the number of degrees of freedom.
 There is no difference between continuous and discontinuous approximations
 with respect to convergence properties.
 However, the discontinuous approximations show better stability properties than the continuous ones in the convection-dominated cases, 
 which will be presented in Section 6. 
 }
\begin{lemma} \label{lem:proj}
  Under the assumption (H3), for all $\bm v = \{ v , v|_{\Gamma_h}\}$ with $v \in H^{k+1}(\Omega)$,
  there exist positive constants $C$ independent of $\varepsilon$ and $h$ such that 
  \begin{eqnarray}
  &\norm{\bm v - \bm P_h \bm v}_d \le C\varepsilon^{1/2} h^k |v|_{k+1},\\
  &\norm{\bm v - \bm P_h \bm v}_{rc} \le Ch^{k+1/2} |v|_{k+1}, \\
  &\norm{\bm v - \bm P_h \bm v}_{rc,*} \le Ch^{k+1/2} |v|_{k+1}. 
  \end{eqnarray}
  \end{lemma}
\begin{proof}
This follows immediately from the definitions. \qed
\end{proof}

\subsection{Inequalities} 
In this section, we quote several useful inequalities for error analysis without proof.
Refer to  \cite{BS08} for the proofs.
\begin{theorem}  \label{ineq}
Let $K \in \mathcal T_h$ and $e$ be an edge of $K$. 
\begin{enumerate}
\item (Trace inequality) There exists a constant $C$ independent of $K$ and $e$ such that
    \begin{eqnarray}  \label{trace-ineq}
       \|v\|_{0,e} \le Ch_e^{-1/2} \left(\|v\|_{0,K}^2 + h_K^2|v|_{1,K}^2 \right)^{1/2}
   \qquad \forall v \in H^1(K).
    \end{eqnarray}  
\item (Inverse inequality) There exists a constant $C$ independent of $K$  such that
  \begin{eqnarray}  \label{inverse-ineq}
       |v_h|_{1,K} \le Ch_K^{-1} \|v_h\|_{0,K} \qquad \forall v_h \in U_h.
   \end{eqnarray}  
\end{enumerate}
\end{theorem}

\section{A hybridized DG method}
\subsection{Formulation}
Now, we are ready to show our hybridized DG method.
 We first state  our formulation:
Find $\bm u_h \in \bm V_h$ such that
\begin{eqnarray}  \label{hdg-cdr}
   B_h(\bm u_h, \bm v_h) 
  = (f,v_h)_\Omega + \ang{g_N, v_h}_{\Gamma_N} \quad \forall \bm v_h \in \bm V_h,
\end{eqnarray}
where
\begin{eqnarray}   
  && B_h(\bm u_h, \bm v_h) := B_h^d(\bm u_h, \bm v_h) + B_h^{rc}(\bm u_h, \bm v_h), \\
  && B^d_h(\bm u_h, \bm v_h)  =   \varepsilon \sum_{K\in\T_h} \bigg[(\nabla u_h, \nabla v_h)_K  \label{hdg-d} 
      + \ang{\frac{\partial u_h}{\partial n}, \hat v_h - v_h}_{\partial K\setminus\Gamma_N}  \\
  && \qquad \qquad \qquad \qquad
    +\ang{\frac{\partial v_h}{\partial n}, \hat u_h - u_h}_{\partial K\setminus\Gamma_N} 
    + \sum_{e \subset \partial K\setminus \Gamma_N} \frac{\eta_e}{h_e} \ang{\hat u_h -  u_h, \hat v_h - v_h}_e 
   \bigg], \nonumber \\
  && B_h^{rc}(\bm u_h, \bm v_h) \label{hdg-rc}
     = \sum_{K\in\T_h}\bigg[ (\bm b \cdot \nabla u_h + cu_h, v_h)_K    \\
  && \qquad \qquad \qquad \qquad  + \ang{\hat u_h -  u_h, [\bm b \cdot \bm n]_+ \hat v_h - [\bm b\cdot \bm n]_- v_h}_{\partial K\setminus\Gamma_N} \bigg], \nonumber\\
  && (f,v_h)_\Omega = \int_\Omega f v_h dx, \\
  &&  \ang{g_N, v_h}_{\Gamma_N} = \int_{\Gamma_N} g_N v_h ds.
\end{eqnarray}
Here $\eta_e$ is a penalty parameter with $\eta_e \ge \eta_{{\it min}} > 0$, 
 $h_e$ is the length of an edge $e$, and the brakets $[\,\cdot\,]_+$ and $[\,\cdot\,]_-$ appearing in \eqref{hdg-rc} denote functions satisfying for some constant $\gamma>0$, 
\begin{eqnarray} \label{eq:bracket}
   [x]_+ + [x]_- \ge  \gamma |x|  \quad \textrm{ and }  \quad [x]_+ - [x]_- = x\qquad (x \in \mathbb R).
\end{eqnarray}
As such functions, we can take 
\begin{equation}
    [x]_+ = \max(0,x), \qquad [x]_-  = \max(0,-x).
\end{equation}
Note that, for all $x \in \mathbb{R}$, it follows that
\begin{equation} \label{upflux}
     [x]_+ + [x]_- = |x|, \qquad  [x]_+ - [x]_- = x.
\end{equation}
The last term in the right-hand side of \eqref{hdg-rc} is an upwinding term which makes
 the coercivity of $B_h^{rc}(\cdot,\cdot)$ hold.
\subsection{Derivation}
\rev{Before proceeding to the analysis of the scheme \eqref{hdg-cdr}, we show how to obtain it.
Multiplying both sides of \eqref{cdr1a} by a test function $v \in H^2(\mathcal T_h)$ and integrating them over the element $K$, we have, after integrating by parts and after adding over all the elements $K \in \mathcal T_h$, }
\begin{eqnarray}
   \sum_{K\in\T_h} \bigg[\varepsilon (\nabla u, \nabla v)_K - \varepsilon\ang{\frac{\partial u}{\partial n}, v}_{\partial K\setminus\Gamma_N} 
    & &+ (\bm b\cdot \nabla u + cu, v)_K \bigg]   \label{diff1} \\
   & &  = (f,v)_\Omega + \ang{g_N,v}_{\Gamma_N}. \nonumber
\end{eqnarray}
We denote the diffusive part and  convective part in \eqref{diff1}  by $D(\cdot,\cdot)$ and $C(\cdot,\cdot)$, respectively, i.e., 
\begin{eqnarray}
   D(u,v) &:=& \varepsilon\sum_{K\in\T_h} \left[ (\nabla u, \nabla v)_K - \ang{\frac{\partial u}{\partial n}, v}_{\partial K\setminus\Gamma_N}\right],  \label{diff-part} \\
   C(u,v) &:=& \sum_{K\in\T_h} (\bm b\cdot \nabla u + cu, v)_K. \label{conv-part}
\end{eqnarray}
We first derive our formulation of the diffusive part. From the continuity of the flux, we have
\begin{eqnarray}
  \label{diff2}
  \sum_{K\in\T_h} \ang{\frac{\partial u}{\partial n}, \hat v}_{\partial K\setminus\Gamma_N} = 0 
  \qquad \forall \hat v \in \hat L^2_D(\Gamma_h).
\end{eqnarray}
Adding \eqref{diff2} to \eqref{diff-part} yields 
\begin{eqnarray}
  \label{diff3}
  D(\bm u,\bm v) = \varepsilon\sum_{K\in\T_h} \left[ (\nabla u, \nabla v)_K + \ang{\frac{\partial u}{\partial n}, \hat v-v}_{\partial K\setminus\Gamma_N} \right],
\end{eqnarray}
where $\bm u = \{ u, \hat u\}$ and $\bm v = \{v,\hat v\} \in \bm V_h$.
Symmetrizing \eqref{diff3} and adding the following penalty term
\begin{eqnarray}
  \label{penalty}
  \sum_{K\in\T_h} \sum_{e \subset \partial K\setminus\Gamma_N} \ang{\frac{\eta_e}{h_e}(\hat u - u), \hat v-v}_{e},
\end{eqnarray}
we obtain $\eqref{hdg-d}$.

Next, we derive the formulation of the convective part. 
Let $\alpha$ and $\beta$ be coefficients to be determined later, and we consider the following form:
\begin{eqnarray}
  C_h(\bm u_h,\bm v_h) &:=& \sum_{K\in\T_h} \Big[ (\bm b\cdot \nabla u_h + cu_h, v_h)_K  \\
               & & \quad +\ang{\hat u_h -u_h, \alpha \hat v_h - \beta v_h}_{\partial K\setminus\Gamma_N} \Big]. \nonumber
\end{eqnarray}
The coefficients  $\alpha$ and $\beta$ are chosen to satisfy the coercivity of $C_h(\cdot,\cdot)$, namely, for some constant $C_c^{rc}>0$,  
\begin{equation} \label{coerc0}
  C_h(\bm v_h, \bm v_h)  \ge C_c^{rc} \norm{\bm v_h}_{rc} \quad \forall \bm v_h \in \bm V_h. 
\end{equation}
The left-hand side in the above can be rewritten as follows:
\begin{eqnarray}
 &&C_h(\bm v_h, \bm v_h) =  \\
 &&\sum_{K\in\T_h}\bigg[ (\rho v_h,v_h)_K + \ang{( \frac 1 2(\bm b\cdot \bm n)  + \beta)v_h,v_h}_{\partial K\setminus\Gamma_N} \nonumber \\
  &&      \qquad \qquad  -\ang{(\alpha+\beta)v_h,\hat v_h}_{\partial K\setminus\Gamma_N} 
                +\ang{\alpha \hat v_h, \hat v_h}_{\partial K\setminus\Gamma_N}\bigg] \nonumber
\end{eqnarray}
for any $\bm v_h \in \bm V_h$, where $\rho$ is the function defined in \eqref{rho}.
 We can find the following conditions to be satisfied
\begin{equation} \label{eq:ab}
    \alpha + \beta  = (\bm b \cdot \bm n) + 2\beta \ge \gamma |\bm b \cdot \bm n|,
\end{equation}
from which it follows that
\begin{equation} \label{upw1}
 \alpha+\beta \ge \gamma |\bm b \cdot \bm n|, \quad  \alpha - \beta = \bm b \cdot \bm n.
\end{equation}
We rewrite the coefficients as  $\alpha = [\bm b \cdot \bm n]_+$ and  $\beta =  [\bm b \cdot \bm n]_-$.
Thus we obtain our formulation \eqref{hdg-cdr}.

\subsection{Relation with other HDG schemes}

\rev{
 As mentioned in the Introduction, our scheme is essentially same as in \cite{Wells11}.
We here remark on the relation between the schemes proposed in \cite{NPC09,CCToAppear} and ours.
In their method,
eliminating the auxiliary variable $\bm q_n$ and taking the stabilization parameter
  $
    \tau = \varepsilon \eta_e/h_e + \max(0,\bm b \cdot \bm n)
$
on each edge $e$, we obtain the almost same scheme as ours.
}
\subsection{Local conservativity}
Let $K$ be an element of $\T_h$, and let $\chi_K$ denote a characteristic function on $K$.
Taking $\bm v_h = \{\chi_K,0\}$ in \eqref{hdg-cdr},
we see that our hybridized method satisfies a local conservation property, i.e.
\begin{equation} \label{loccons}
     \int_K (\bm b \cdot \nabla u_h + cu_h) dx  -\int_{\dK} \hat{\bm\sigma}(\bm u_h) \cdot \bm n ds
      = \int_K f dx + \int_{\Gamma_N \cap \partial K} g_N ds,
\end{equation}
where  $\hat{\bm \sigma}$ is an upwind numerical flux,  defined as follows:
\[
     \hat{\bm \sigma}(\bm u_h) := \varepsilon (\nabla u_h + \frac{\eta_e}{h_e} (\hat u_h- u_h)\bm n)
      + [\bm b \cdot \bm n]_-(\hat u_h -u_h)\bm n.
\]
This property is appropriate in a convection-diffusion regime.
Note that the conforming finite element method does not possess such a property in general.   

\section{Error analysis}
In this section, we shall establish error estimates for \eqref{hdg-cdr}.
\begin{lemma} \label{lem:diff}
 For the bilinear form corresponding to the diffusive part, we have the following properties.
\begin{enumerate}
\item (Boundedness)
     There exists a constant $C_b^d>0$ such that
   \begin{eqnarray} \label{bdd-d} 
|B_h^d(\bm w,\bm v)| \le C_b^d \norm{\bm{w}}_d  \norm{\bm{v}}_d 
     \quad \forall \bm w, \bm v \in \bm V.
   \end{eqnarray}
\item (Coercivity)  There exists a constant $C_c^d>0$ such that
   \begin{eqnarray}  \label{cc-d}
          B_h^d(\bm v_h,\bm v_h) \ge C_c^d \norm{\bm v_h}^2_d \quad \forall \bm v_h \in \bm V_h.
   \end{eqnarray}
   \noindent    Here $C_b$ and $C_c^d$ are independent of $\varepsilon$ and $h$.
\end{enumerate}
\end{lemma}
\begin{proof}
We first prove the boundedness.
Applying the Schwarz inequality to each term of \eqref{hdg-d}, we have
\begin{eqnarray} \label{bdd1}
        |B_h^d(\bm w, \bm v)|
        &\le& \varepsilon \sum_{K\in\T_h}\Bigg[
            \|\nabla w\|_{0,K}\|\nabla v\|_{0,K}   \\
        &  &\quad+  \sum_{e \subset \partial K\setminus\Gamma_N}\bigg(\left\| \frac{\partial w}{\partial n} \right\|_{0,e}\|\hat v-v\|_{0,e}
            +  \left\| \frac{\partial v}{\partial n} \right\|_{0,e}\|\hat w-w\|_{0,e} \nonumber \\
        & & \qquad \qquad \qquad +  \left\|\sqrt\frac{\eta_e}{h_e} (\hat w-w)\right\|_{0,e}        \left\|\sqrt\frac{\eta_e}{h_e}(\hat v-v)\right\|_{0,e}\bigg)
            \Bigg]. \nonumber
\end{eqnarray}
From the trace theorem, we have
\begin{eqnarray} \label{bdd2}
      \left\| \frac{\partial w}{\partial n} \right\|_{0,e} \le C h_e^{-1/2} (|w|_{1,K}^2+h_K^2|w|_{2,K}^2)^{1/2}.
\end{eqnarray}
From \eqref{bdd1}, \eqref{bdd2} and the Cauchy-Schwarz inequality, it follows that
\begin{eqnarray*}
      \label{bbd3}
|B_h^d(\bm w, \bm v)|   \le \max(1+C\eta_{{\it min}}^{-1/2}, 2) \norm{\bm w}_d \norm{\bm v}_d.
\end{eqnarray*}
 Next, we prove the coercivity. By definition, 
\begin{eqnarray}\label{cc1}
     B_h^d(\bm v_h, \bm v_h) &\ge& |v_h|_{1,h}^2  +|\bm v_h|_{j,h}^2 \\ 
       && \quad  -2\sum_{K\in\T_h} \left(\sum_{e\subset \dK\setminus\Gamma_N}\left\| 
       \frac{\partial v_h}{\partial n} \right\|_{0,e}\|\hat v_h-v_h\|_{0,e}  \right). 
       \nonumber
\end{eqnarray}
By the trace theorem, the inverse and Young's inequalities, we have for any $\delta \in (0,1)$,
\begin{eqnarray} \label{cc2}
         \left\| \frac{\partial v_h}{\partial n} \right\|_{0,e} \|\hat v_h - v_h\|_{0,e}
        & \le& \frac{C}{h_e} |v_h|_{1,K} \|\hat v_h - v_h\|_{0,e}  \\
        & \le& \frac{C}{2\delta\eta_e}|v_h|_{1,K}^2 + \frac{\delta}{2} \left\|\sqrt \frac{\eta_e}{h_e} (\hat v_h- v_h)\right\|_{0,e}^2    \quad   \forall v_h \in U_h. \nonumber
\end{eqnarray}
 From \eqref{cc1} and \eqref{cc2}, we obtain 
\begin{eqnarray}
    \label{cc3}
     B_h^d(\bm v_h, \bm v_h) \ge 
         \left( 1 - \frac{C}{\delta\eta_{{\it min}}}  \right) |v_h|_{1,h}^2 + \left(1 -  \delta\right) |\bm v_h|_{j,h}^2. 
\end{eqnarray}
If $\eta_{{\it min}} > 4C$, then we can take $\delta = \sqrt{C/\eta_{{\it min}}}< 1/2$, which implies that 
\[ 
  1 - \frac{C}{\delta\eta_{{\it min}}} > 1/2, \quad  1- \delta > 1/2.
\]
Hence we have   
\begin{eqnarray*}
      B_h^d(\bm v_h, \bm v_h) &\ge& \frac 1 2 (|v_h|_{1,h}^2 + |\bm v_h|_{j,h}^2)
   =:  \frac 1 2 \norm{\bm v_h}_{d,h}^2.
\end{eqnarray*}
Since the norms $\norm\cdot_d$ and $\norm{\cdot}_{d,h}$
 are equivalent to each other over the finite dimensional space $\bm V_h$, we obtain the coercivity \eqref{cc-d}. \qed
\end{proof}
\begin{lemma} \label{lem:rc}
For the bilinear form corresponding to the convective part, we have the following properties.
\begin{enumerate}
\item There exists a constant $C_b^{rc}>0$ such that for all $\bm v \in \bm V$, $\bm w_h \in \bm V_h$,
   \begin{eqnarray} \label{bdd-rc}
      |B_h^{rc}(\bm v - \bm P_h \bm v,\bm w_h)| \le C_b^{rc}\norm{\bm v - \bm P_h\bm v}_{rc,*} 
           \norm{\bm w_h}_{rc}. \quad 
   \end{eqnarray}
\item (Coercivity) There exists a constant $C_c^{rc}>0$ such that
   \begin{eqnarray} \label{cc-rc} 
      B_h^{rc}(\bm v_h,\bm v_h) \ge C_c^{rc} \norm{\bm v_h}^2_{rc} \quad \forall \bm v_h \in \bm V_{h}.
   \end{eqnarray}
\end{enumerate}
  Here $C_b^{rc}$ and $C_c^{rc}$ are independent of $h$.
\end{lemma}
\begin{proof}
Set $\bm z = \bm v - \bm P_h \bm v$.  
By Green's formula, we have
\begin{eqnarray*}
      B_h^{rc}(\bm z, \bm w_h) 
       &=& \sum_{K\in\T_h} (z ,(-\bm b) \cdot\nabla  w_h)_K + 
              \sum_{K\in\T_h}   ((c - \mathrm{div } \bm b) z, w_h)_K  \\ 
      & &  + \sum_{K\in\T_h} \left( \ang{(\bm b\cdot \bm n)z,w_h}_{\partial K}
           + \ang{\hat z -   z, [\bm b\cdot \bm n]_+ \hat w_h - [\bm b\cdot \bm n]_-w_h}_{\partial K\setminus\Gamma_N}\right) \\
      & =:& \textit{I} + \textit{II} + \textit{III}. \nonumber
\end{eqnarray*}
For any $K \in \T_h$, we have 
\begin{eqnarray*}
   (z , (-\bm b) \cdot\nabla w_h)_K 
       &=& (z,   (\bm P_h^0  \bm b - \bm b) \cdot\nabla w_h)_K - (z, (\bm P_h^0 \bm b)\cdot \nabla w_h)_K. 
\end{eqnarray*}
 From the property of the projection $\bm{P}_h^0$, it follows that the second term in the right-hand side of the above 
 equality vanishes.
Using the inverse inequality, we get
\begin{eqnarray} \label{bdd-rc2}
  |I| &\le&  \sum_{K\in\T_h} |(z, (\bm P_h^0  \bm b - \bm b) \cdot\nabla w_h)_K|  \\
       &\le&  C |\bm b|_{1,\infty} \sum_{K\in\T_h} \left(
       \|z\|_{0,K}\|w_h\|_{0,K}\right). \nonumber
\end{eqnarray}
  By using the Schwarz inequality, we have
\begin{eqnarray} \label{bdd-rc3}
  |\textit{II}|
    &\le C(|c|_{0,\infty} + |\bm b|_{1,\infty}) \sum_{K\in\T_h}(\|z\|_{0,K} \|w_h\|_{0,K}).
\end{eqnarray}
 To estimate the bound of $\textit{III}$, we rewrite it as follows: 
\begin{eqnarray} 
   \textit{III} 
   &=& \sum_{K\in\mathcal T_h} 
   \Big(\ang{(\bm b \cdot \bm n)z, w_h-\hat w_h}_{\partial K\setminus\Gamma_N} \\
   &&  \qquad+\ang{\hat z- z,[\bm b\cdot\bm n]_-(\hat w_h-w_h)}_{\partial K\setminus\Gamma_N}\Big) \label{bdd-rc6} \\
   & & + \ang{(\bm b\cdot\bm n)z, w_h}_{\Gamma_N}. \nonumber 
\end{eqnarray}
   From the property of the $L^2$-projection $\hat P_h^0$, we have 
\begin{eqnarray} \label{bdd-rc5}
       |\ang{(\bm b\cdot\bm n)z, w_h}_{\Gamma_N}|
       &=& 
       |\ang{(\bm b\cdot\bm n- \hat P_h^0(\bm b\cdot \bm n))z,w_h}_{\Gamma_N}| \\
       &\le& C|\bm b|_{1,\infty} \left( \sum_{K \in \T_h, 
         \dK \cap \Gamma_N \neq \emptyset}
       \|z\|^2_{0,\dK} \right)^{1/2} \cdot  \|w_h\|_{0,\Omega}. \nonumber
\end{eqnarray}
Here we have used the property of the $L^2$-projection $\bm P_h$, the trace theorem and
 the inverse inequality. Therefore we get
\begin{eqnarray} \label{bdd-rc4}
  |\textit{III}| \le C\|\bm b\|_{1,\infty} 
 \norm{z}_{rc,*} \norm{w_h}_{rc}.
\end{eqnarray}
From \eqref{bdd-rc2}, \eqref{bdd-rc3} and \eqref{bdd-rc4}, we obtain \eqref{bdd-rc}.

We now turn to the proof of (2).
By Green's formula, we have
\begin{eqnarray*}
  B_h^{rc}(\bm v_h, \bm v_h) 
  &=& \sum_{K\in\T_h} \Big( \int_K (c - \frac 1 2 \mathrm{div} \bm  b) v_h^2 dx
        \\
  &&   + \left.
  \frac 1 2  \int_{\partial K} (\bm b\cdot\bm n) v_h^2dx 
   + \ang{\hat v_h-v_h, [\bm b\cdot\bm n]_+\hat v_h - [\bm b\cdot\bm n]_-v_h}_{\dKN} \right) \\
  & =:& \textit{IV} + \textit{V}.
\end{eqnarray*}
By the assumption \eqref{rho}, we have
$
   \textit IV \ge \rho_0 \|v_h\|_{0,\Omega}^2.
$
 Since
\[
     \ang{(\bm b\cdot \bm n)v_h, v_h }_{\Gamma_N} \ge 0,
\]
we have
\begin{eqnarray} \label{eq10}
 \textit{V}
  &\ge& \sum_{K \in \T_h}
      \Big( \frac 1 2 \ang{(\bm b\cdot\bm n) v_h, v_h}_{\dKN}  \\
  & &
      \qquad+ \ang{\hat v_h-v_h, [\bm b\cdot\bm n]_+\hat v_h - [\bm b\cdot\bm n]_- v_h}_{\dKN} \Big) \nonumber \\
  &=&  \frac 1 2 \sum_{K\in\T_h}\Big( 
          \ang{([\bm b\cdot\bm n]_++[\bm b\cdot\bm n]_-) (\hat v_h -  v_h), \hat v_h - v_h}_{\dKN}\nonumber \\
  & &\qquad +  \ang{([\bm b\cdot\bm n]_+ - [\bm b\cdot\bm n]_-) \hat v_h, \hat v_h}_{\dKN} 
  \Big)\nonumber \\
  &\ge& \frac 1 2 \sum_{K\in\T_h}\Big( \ang{\gamma |\bm b\cdot\bm n|(\hat v_h - v_h), \hat v_h - v_h}_{\dKN}
      + \ang{(\bm b\cdot\bm n)\hat v_h , \hat v_h}_{{\dKN}}\Big). \nonumber
\end{eqnarray}
From the continuity of the flux, that is,
\[ 
\sum_{K\in\T_h} \ang{(\bm b\cdot\bm n)\hat v_h , \hat v_h}_{{\dKN}} = 0,
\]
we have
\begin{eqnarray*}
 \textit{V} \ge  \gamma\sum_{K\in\T_h} \| |\bm b \cdot \bm n|^{1/2} (\hat v_h - v_h)\|^2_{0,{\dKN}}.
\end{eqnarray*}
 Thus we obtain the coercivity \eqref{cc-rc}. \qed
\end{proof}

Now let us state our main result.  
\begin{theorem} \label{lem:cdr} 
\begin{enumerate}
\item (Galerkin orthogonality)
    Let $u$ be the exact solution of \eqref{cdr}, and set $\bm u = \{u,u|_{\Gamma_h}\}$. 
    Let $\bm u_h$ be the approximate solution by \eqref{hdg-cdr}. Then we have
    \begin{eqnarray*}
        \label{go}
         B_h(\bm u - \bm u_h,\bm v_h) = 0 \quad \forall \bm v_h \in \bm V_h.    
    \end{eqnarray*}
\item We assume (H3). Then, there exists a constant $C_b$ independent of $h$ and $\varepsilon$ such that
    \begin{eqnarray*}
       |B_h(\bm v - \bm P_h \bm v,\bm w_h)| \le C_b \norm{\bm v - \bm P_h\bm v}_* \norm{\bm w_h} \qquad  \bm v \in \bm V, \bm w_h \in \bm V_h. 
    \end{eqnarray*}
\item (Coercivity) There exists a constant $C_c$ independent of $h$ and $\varepsilon$ such that
   \begin{eqnarray}\label{cc}
     B_h(\bm v_h,\bm v_h) \ge C_c \norm{\bm v_h}^2 \qquad \bm v_h \in \bm V_{h}.
   \end{eqnarray}
\end{enumerate}
\end{theorem}

\begin{proof}
Combining Lemma \ref{lem:diff} and Lemma \ref{lem:rc}, we can easily get (2) and (3).
The Galerkin orthogonality (1) follows immediately from consistency, i.e.,
\begin{equation}
\label{cons}
  B_h(\bm u, \bm v_h) = (f,v_h) + \ang{g_N, v_h}_{\Gamma_N} \quad \forall \bm v_h \in \bm V_h.
\end{equation}
Noting that $u - \hat u = 0$ on $\Gamma_h$ and
\[
   \sum_{K\in\T_h}\ang{\frac{\partial u}{\partial n}, \hat v_h}_{\partial K\setminus \Gamma_N} = 0
\]
 for any $\hat v_h \in \hat U_h$, we have \eqref{cons}. \qed
\end{proof}
\begin{remark}
(Local solvability)
Let $\bm u_h = \{u_h, \hat u_h\}$ be an approximate solution of our HDG scheme \eqref{hdg-cdr}.
We point out that $u_h|_{K}$ can be determined by only $\hat u_h|_{\dK}$ for each $K \in \T_h$.
Taking $\bm v_h = \{v_h|_{K}, 0\}$ in \eqref{hdg-cdr}, we have
\begin{eqnarray} \label{ls01}
     && B_h(\{u_{h}|_{K}, 0\}, \{v_{h}|_{K},0\}) \\
     && \quad = (f,v_{h}|_{K})_K + \ang{g_N, v_{h}|_{K}}_{\Gamma_N\cap \dK}
         - B_h(\{0,\hat u_{h}|_{\dK}\}, \{v_{h}|_{K}, 0\}). \nonumber
\end{eqnarray}
Note that the support of $v_h$ is included in $K$.
In \eqref{cc}, choosing $\bm v_h = \{u_h|_K, 0\}$, we obtain 
\[
    B_h(\{ u_h|_K, 0\}, \{ u_h|_K, 0\}) \ge C_c \norm{\{u_h|_K,0\}}^2,  
\]
which implies that the discrete system of linear equations \eqref{ls01} is regular, that is,
$u_h|_K$ exists uniquely.  Hence we can eliminate $u_h|_K$ by means of $\hat u_h|_{\dK}$.
Consequently, only the inter-element unknown $\hat u_h$ remains in \eqref{hdg-cdr}, which
is the mechanism of the reduction of the size of the resulting matrix.  
\end{remark}
As a consequence of Theorem \ref{lem:cdr}, we obtain a priori error estimates of optimal order in the HDG norm.
\begin{theorem} \label{thm:err}
 Let $u$ be the exact solution of \eqref{cdr}, and let $\bm u = \{u,u|_{\Gamma_h}\}$. 
 Let $\bm u_h$ be the approximate solution by \eqref{hdg-cdr}.
 Recall that we are assuming (H1), (H2) and (H3).
 If $u \in H^{k+1}(\Omega)$, then  we have the following error estimate:
  \begin{eqnarray}
    \label{eq:err1}
    \norm{\bm u - \bm u_h}  \le C(\varepsilon^{1/2} + h^{1/2})h^k |u|_{k+1},
  \end{eqnarray}
where $C$ denotes a positive constant independent of $\varepsilon$ and  $h$.
\end{theorem}
\begin{proof}
By using the three properties in Lemma \ref{lem:cdr}, we deduce
  \begin{eqnarray*}   
 C_c \norm{\bm u_h - \bm P_h \bm u}^2  
  &\le& B_h(\bm u_h - \bm P_h \bm u,\bm u_h - \bm P_h \bm u) \\ 
  & =&  B_h(\bm u - \bm P_h  \bm u, \bm u_h - \bm P_h  \bm u)  \\  
  &\le&  C_b 
       \norm{\bm u - \bm P_h \bm u}_*  \norm{ \bm u_h - \bm P_h \bm u}.
  \end{eqnarray*}
 Hence we have 
\[
    \norm{\bm u_h - \bm P_h \bm u} \le \frac{C_b}{C_c}\norm{\bm u - \bm P_h \bm u}_*.
\]
By the triangle inequality and Lemma \ref{lem:proj}, we have
  \begin{eqnarray*}
   \norm{\bm u -  \bm u_h}
  & \le & \norm{\bm u_h - \bm P_h \bm u} + \norm{\bm u - \bm P_h \bm u} \\
   & \le &(1+\frac{C_b}{C_c}) \norm{\bm u - \bm P_h \bm u}_* \\
   & \le &C( \varepsilon^{1/2}  +  h^{1/2}) h^k |u|_{k+1},
  \end{eqnarray*}
which completes the proof. \qed
\end{proof}
 
\section{The relation between $\bm u_h$ and  the solution of the reduced problem}
Let $u_0$ be the solution of the reduced problem of \eqref{cdr} :
\begin{subequations}
 \label{reduced-problem}
\begin{eqnarray}
     \bm b \cdot \nabla u_0 + cu_0 &=& f  \textrm{ in } \Omega, \\
       u_0 &=& 0   \textrm{ on } \Gamma_D. 
\end{eqnarray} 
\end{subequations}
 Here we assume that $\Gamma_D = \Gamma_-$, $g_N \equiv 0$, and
 that the existence and uniqueness of  a solution $u_0 \in H^2(\Omega)$ to
 \eqref{reduced-problem}.
 Set $\bm  u_0 = \{ u_0, u_0|_{\Gamma_h}\}$ $\in \bm V_h$.
  The aim of this section is to show an approximate solution $\bm u_h$ gets closer to $\bm u_0$ when $\varepsilon$ tends to 0.
This suggests that our hybridized DG method \eqref{hdg-cdr} is 
stable even when $\varepsilon$ is very small.

\begin{theorem} \label{stab-uheps}
 Let $\bm u_{h}$  be the approximate solution
 defined by \eqref{hdg-cdr}. 
 Then we have the following inequality:
 \begin{eqnarray} \label{stab0} 
  \norm{\bm u_h - {\bm u}_0 } 
 &\le C\left(\norm{ {\bm u}_0}_d + \norm{{\bm u}_0 - \bm P_h {\bm u}_0 }\right),
 \end{eqnarray}
where $C$ is a constant  independent of $\varepsilon$ and $h$.
\end{theorem}
\begin{proof}
By the consistency of $B_h^{rc}(\cdot,\cdot)$, we have
\begin{eqnarray}
 B_h^{rc} (\bm u_0, \bm v_h) = (f, v_h),
\end{eqnarray}
from which it follows that 
\begin{eqnarray}
  B_h({\bm u}_0, \bm v_h) = (f,v_h) + B_h^d(\bm u_0, \bm v_h).  \label{stab2}
\end{eqnarray}
Subtracting \eqref{stab2} from \eqref{hdg-cdr} gives us 
\begin{eqnarray}
  B_h({\bm u}_0 - \bm u_h, \bm v_h)  =  B^d_h(\bm u_0,\bm v_h)
 \quad \forall \bm v_h \in \bm V_h.
 \label{stab3}
\end{eqnarray} 
Choosing $\bm v_h =  \bm u_h - \bm P_h {\bm u}_0 \in \bm V_{h}$ in \eqref{cc}, we have
\begin{eqnarray*}
 & & C_c\norm{\bm u_h - \bm P_h {\bm u}_0 }^2 \\
 & & \quad \le B_h(\bm u_h - \bm P_h {\bm u}_0, \bm u_h - \bm P_h {\bm u}_0) \\
 &&\quad  =  B_h(\bm u_h - {\bm u}_0 , \bm u_h - \bm P_h {\bm u}_0) 
 +B_h({\bm u}_0 - \bm P_h {\bm u}_0, \bm u_h - \bm P_h {\bm u}_0)\\
 &&\quad \le  |B^d_h({\bm u}_0, \bm u_h -\bm P_h {\bm u}_0)| 
   +|B_h({\bm u}_0 - \bm P_h {\bm u}_0, \bm u_h - \bm P_h {\bm u}_0)|\\
 &&\quad \le C_b^d\norm{{\bm u}_0}_d   \norm{\bm  u_h - \bm P_h {\bm u}_0}_d  +  C_b\norm{{\bm u}_0 - \bm P_h {\bm u}_0 }  \norm{ \bm u_h - \bm P_h {\bm u}_0 }.
\end{eqnarray*}
Then we have
\begin{eqnarray*}
 C_c \norm{ \bm u_h -  \bm P_h {\bm u}_0 }
\le  C_b^d\norm{{\bm u}_0 }_d + C_b\norm{{\bm u}_0 - \bm P_h{\bm u}_0 }. 
\end{eqnarray*}
By the triangle inequality, we get
\begin{eqnarray*}
   \norm{ \bm u_h -  {\bm u}_0 }
 &\le& \norm{ \bm u_h -  \bm P_h {\bm u}_0 } +   \norm{{\bm u}_0 -  \bm P_h{\bm u}_0 }  \\
 &\le& \frac{C_b^d}{C_c}\norm{{\bm u}_0}_d + \left(1+\frac{C_b}{C_c}\right)\norm{{\bm u}_0 - \bm P_h{\bm u}_0 }.
\end{eqnarray*}
Since $C_b^d$, $C_c$ and $C_b$ are independent of $\varepsilon$,
 we obtain the inequality \eqref{stab0}. \qed
\end{proof}

\section{Numerical results}
\subsection{The case of smooth solutions}
Let $\Omega$ be the unit square domain, $\bm{b} = (1,1)^T$ and $c \equiv 0$.
Then, the test problem reads
\begin{subequations}\label{ex1}
\begin{eqnarray} 
    -\varepsilon \Delta u + \bm{b} \cdot \nabla u  &=& f  \textrm{ in } \Omega,  \\
                                                u  &=& 0  \textrm{ on } \Gamma_D = \partial\Omega,
\end{eqnarray}
\end{subequations}
where the data $f$ is chosen so that the exact solution is 
\[
   u(x,y) = \sin(\pi x)\sin(\pi y),
\]
which is a smooth function. 
We computed approximate solutions of \eqref{hdg-cdr} with piecewise linear elements
 for different $\varepsilon = 1, 10^{-3}, 10^{-6}$ to confirm that our scheme is valid for
 not only diffusion-dominated cases but also convection-dominated ones.
The meshes we used are uniform triangular ones.
Figure \ref{ss-cdiagrams} shows the convergence diagrams with respect to the $L^2$-norm and $H^1(\T_h)$-seminorm.
It can be observed that the convergence rates are optimal 
for each $\varepsilon$. 
  
\begin{figure}[htbp]
\begin{center}
\begin{tabular}{c}
\includegraphics[width=84mm]{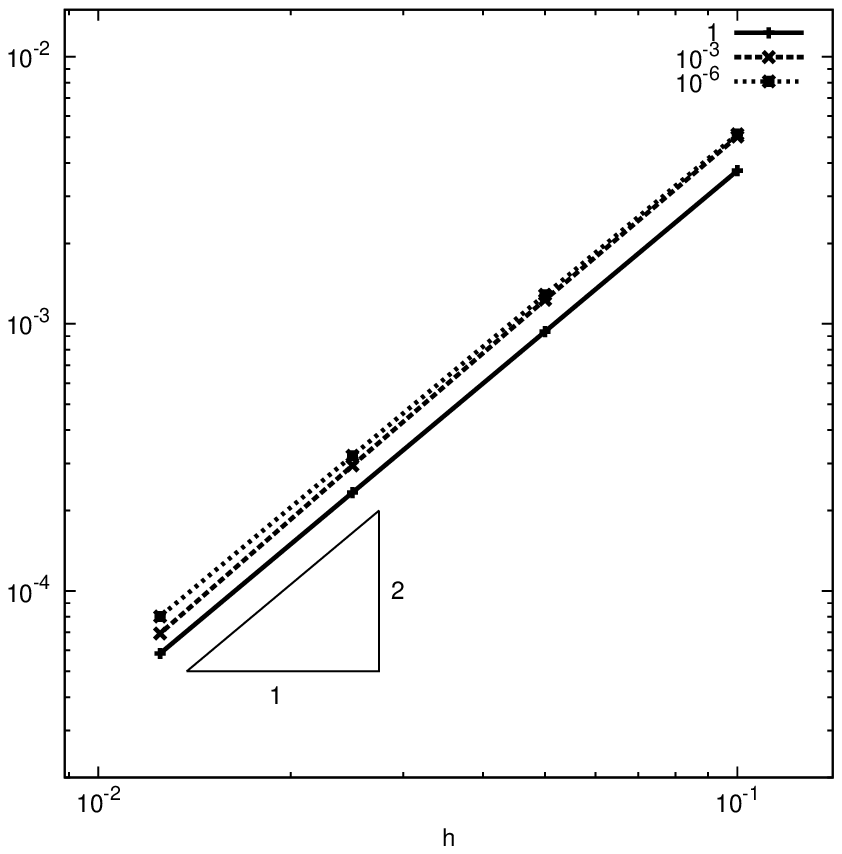}
\\
\includegraphics[width=84mm]{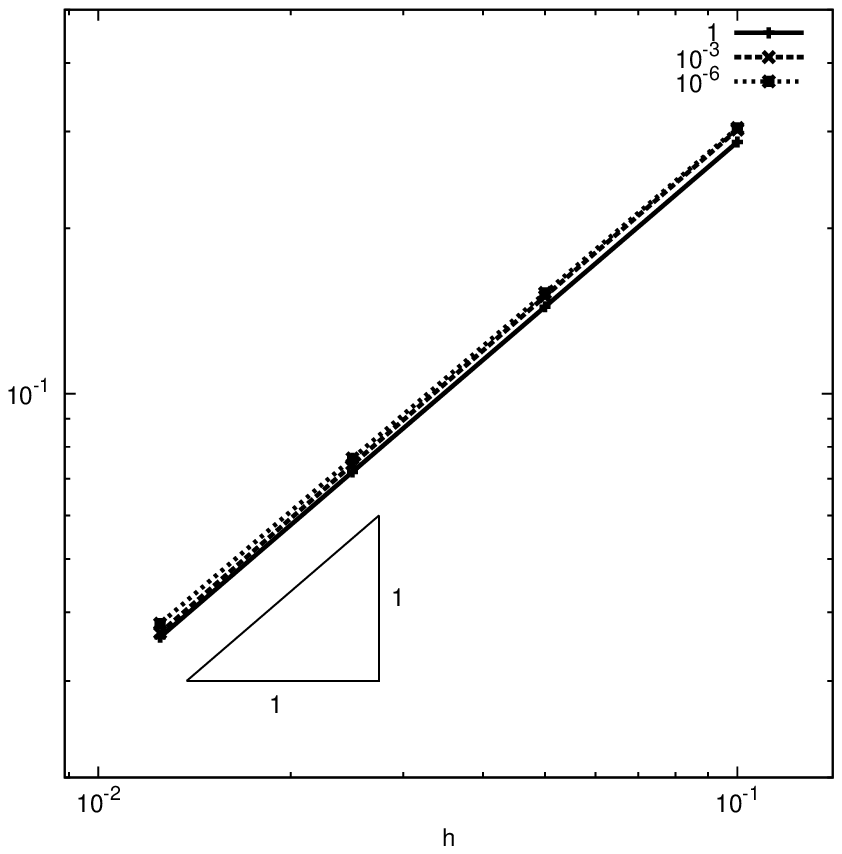}
\end{tabular}
\end{center}
\caption{Convergence diagrams in the case of smooth solutions. $L^2$-errors(top) and $H^1(\mathcal T_h)$-errors(bottom) in $\Omega$ for various $\varepsilon$.} 
\label{ss-cdiagrams}
\end{figure}

\subsection{Boundary layers}
 Next, we show the results in the case that an exact solution has a boundary layer. 
 Let us consider the same problem in the previous, that is, $\Omega=(0,1)^2$, $\bm{b} = (1,1)^T$ and $c \equiv 0$.
 The data $f$ is given so that the exact solution is
\[
   u(x,y) =  \sin(\pi x/2)\sin(\pi y/2) \left(1 - e^{(x-1)/\varepsilon}\right)\left(1 - e^{(y-1)/\varepsilon}\right). 
\]
This solution $u$ has a boundary layer near $x = 1$ or $y=1$.
In Figure \ref{bl-p1-hdg}, we display the graphs of the approximate solutions $u_h$ and $\hat u_h$ for $h = 1/10$.
It can be seen that no spurious oscillation appears 
unlike the classical finite element method.
\rev{
Notice that our approximate solutions $u_h$ and $\hat u_h$ do not exactly capture the boundary layer.
}
For the comparison, we also display the approximate solutions by the streamline upwind Petrov-Galerkin (SUPG) method (see, for example, \cite{BS82,KA03}) in Figure \ref{bl-supg-p1}. 
A stabilization parameter $\tau_K$ is given by
\[
    \tau_K = \frac{h_K}{2\|\bm b\|_{0,\infty,K}}
    \left(\coth Pe_K - \frac{1}{Pe_K}\right),
\]
where $Pe_K$ is the local P\'eclet number, defined by $Pe_K = \|\bm b\|_{0,\infty,K}/(2\varepsilon)$.
We refer to \cite{JK07,Knobloch09} on the parameters of the SUPG method.
From Figure \ref{bl-supg-p1}, it can be seen that  spurious overshoot is observed
near the boundary layer in the SUPG solutions, which suggests that our hybridized scheme is more stable than the SUPG one. 
We emphasize that hybridized schemes do not need any stabilization parameter such as $\tau_K$. In the SUPG method,
 some manipulation of parameters is inevitable.  
Figure \ref{bl-cdiagrams} shows that the convergence diagrams in $\Omega_{0.9} := (0, 0.9)^2$.
We observe that the convergence rates in the $L^2$-norm and $H^1(\mathcal T_h)$-seminorm  are optimal.
\begin{figure}[htbp] 
\begin{center}
\begin{tabular}{c}
\begin{minipage}[t]{.5\textwidth}
\includegraphics[width=\textwidth]{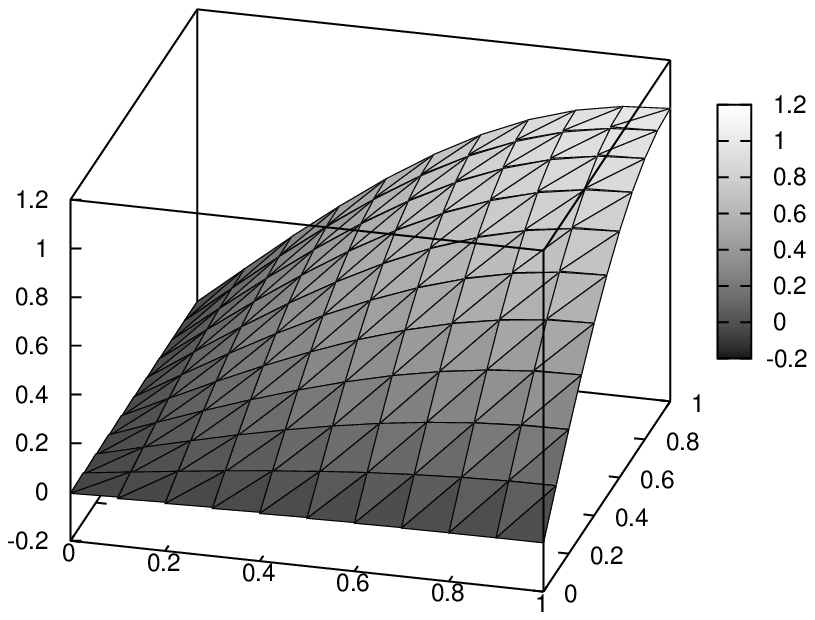}
\end{minipage}
\begin{minipage}[t]{.5\textwidth}
\includegraphics[width=\textwidth]{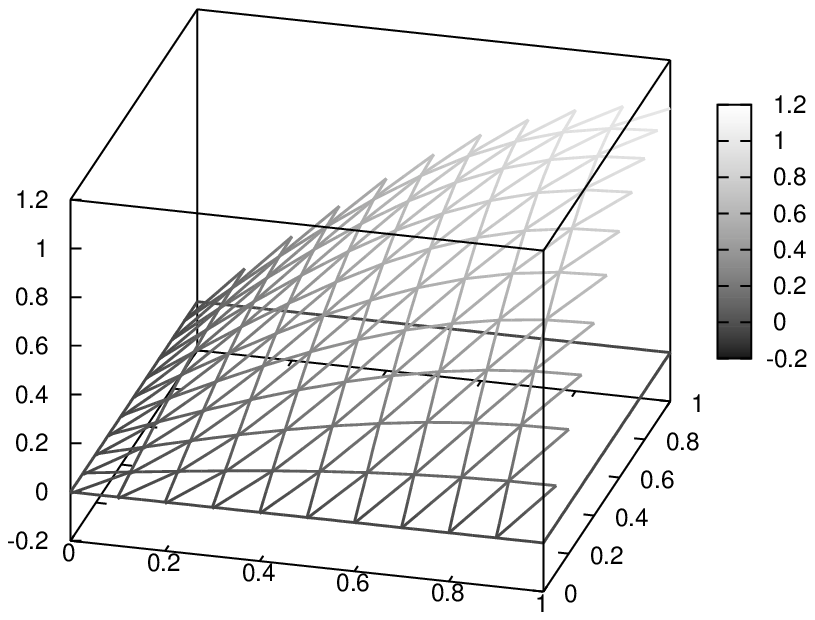}
\end{minipage}
\end{tabular}
\end{center}
\caption{Our HDG solutions $u_h$(left) and $\hat u_h$(right) for $h=1/10$ and  $\varepsilon=10^{-6}$.} 
\label{bl-p1-hdg}
\end{figure}

\begin{figure}[htbp] 
\begin{center}
\begin{tabular}{c}
\begin{minipage}[t]{.5\textwidth}
\includegraphics[width=\textwidth]{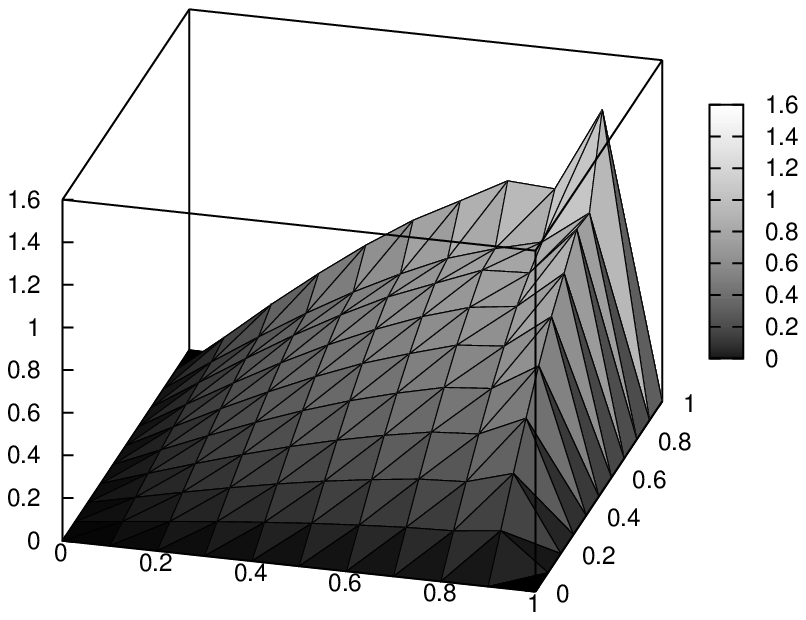}
\end{minipage}
\begin{minipage}[t]{.5\textwidth}
\includegraphics[width=\textwidth]{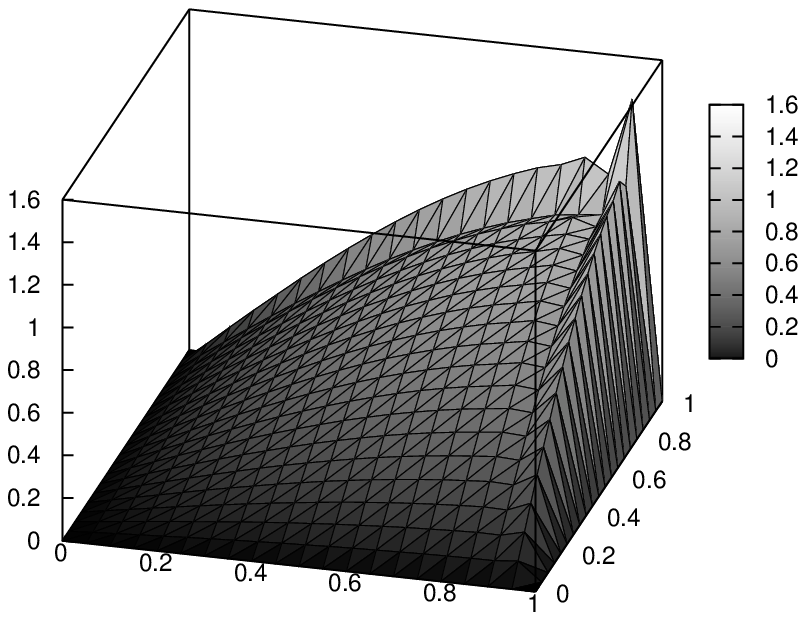}
\end{minipage}
\end{tabular}
\end{center}
\caption{The SUPG solutions for $h=1/10$ (left) and $h=1/20$(right) in the case of  $\varepsilon=10^{-6}$.} 
\label{bl-supg-p1}
\end{figure}

\begin{figure}[htbp]
\begin{center}
\begin{tabular}{c}
\includegraphics[width=84mm]{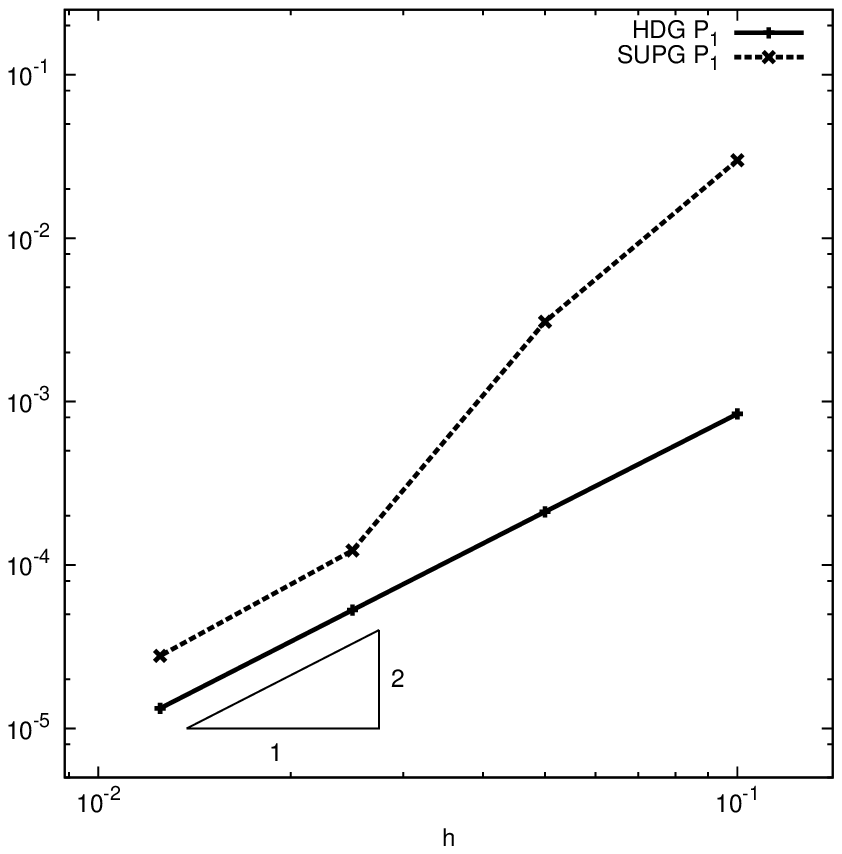}
\\
\includegraphics[width=84mm]{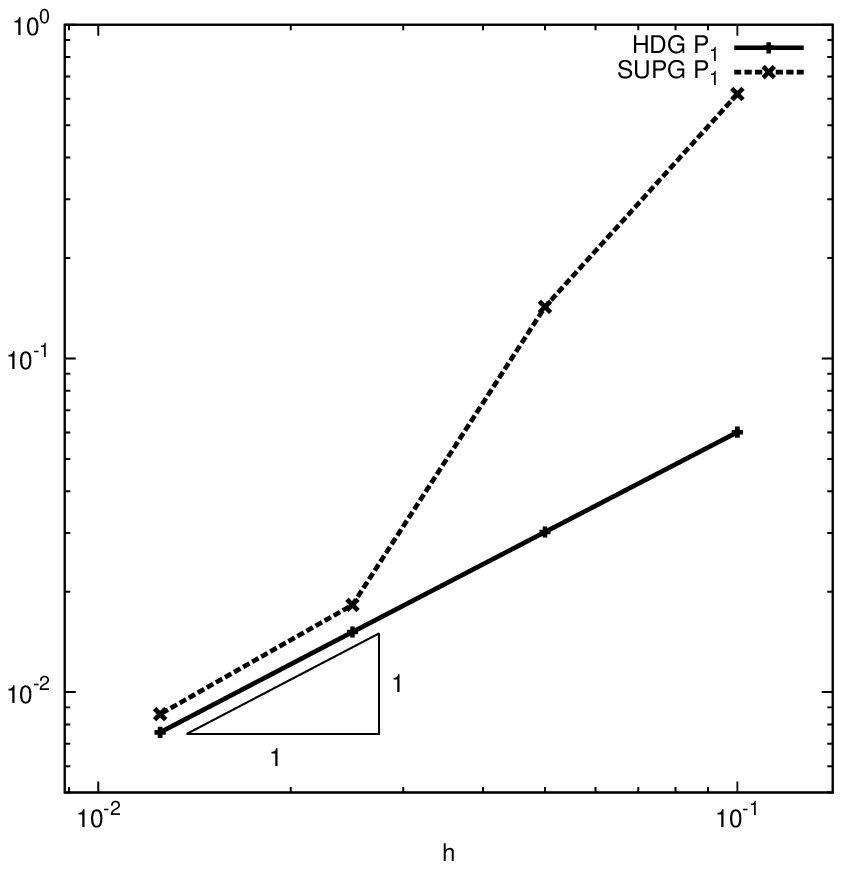}
\end{tabular}
\end{center}
\caption{Convergence diagrams in the boundary layer case. $L^2$-errors(top) and $H^1(\mathcal T_h)$-errors(bottom) in $\Omega_{0.9}$ for $\varepsilon=10^{-6}$.} 
 \label{bl-cdiagrams}
\end{figure}

\subsection{Continuous approximations for $\hat u_h$}
\rev{We will present  numerical results when continuous funcions are employed for $\hat u_h$. 
The same problem in Section 6.2 is considered.
Figure \ref{bl-p1-hdg-conti} displays the graphs of the approximate solutions in the case
$\varepsilon=10^{-6}$ and $h=1/10$.
It can be seen that overshoots  appear near the outflow boundary, which is similar
to the result by the SUPG method. 
This suggests that  discontinuous approximations may have better
stability properties than continuous ones in the convection-dominated cases.} 

\begin{figure}[thbp] 
\begin{center}
\begin{tabular}{c}
\begin{minipage}[t]{.5\textwidth}
\includegraphics[width=\textwidth]{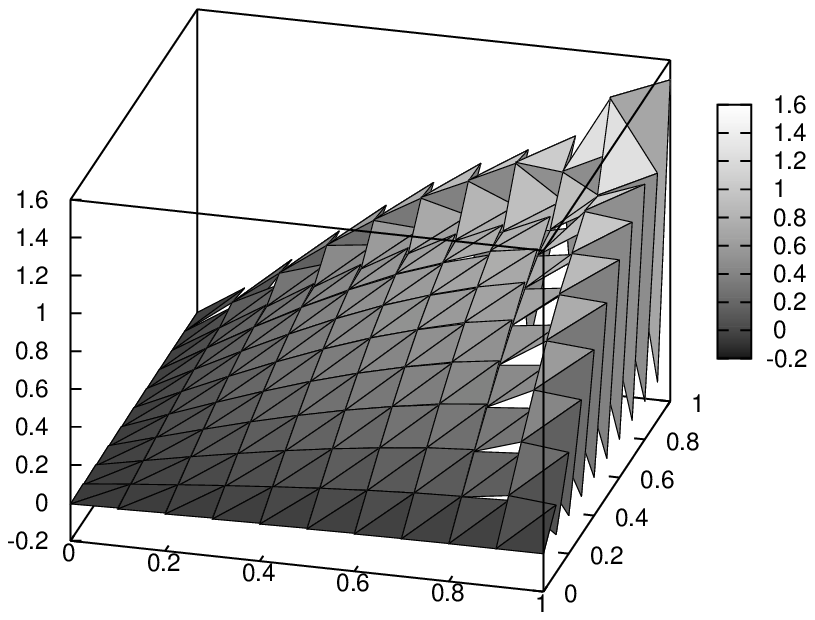}
\end{minipage}
\begin{minipage}[t]{.5\textwidth}
\includegraphics[width=\textwidth]{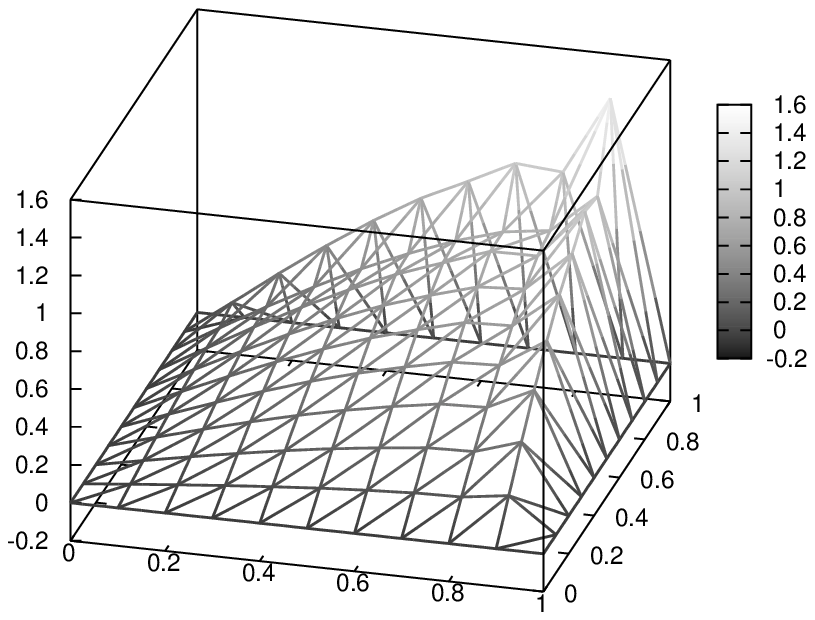}
\end{minipage}
\end{tabular}
\end{center}
\caption{Our HDG solutions  in the case that $\hat u_h$ is continuous: $u_h$(left) and $\hat u_h$(right).  }
\label{bl-p1-hdg-conti}
\end{figure}

\section{Conclusions} 
We have presented a new derivation of a hybridized scheme for the convection-diffusion problems.
In our formulation, an upwinding term is added to satisfy the coercivity of a convective part. 
As a result, our scheme is stable even when $\varepsilon \downarrow 0$.
Indeed, numerical results show that no spurious oscillation appears in our approximate solutions.
We have proved the error estimates of optimal order  in the HDG norm, and
discussed the relation between our approximate solution and an exact one of the reduced problem.

\begin{acknowledgement}
I thank Professors B. Cockburn, F. Kikuchi and N. Saito 
who encouraged me through valuable discussions. 
\end{acknowledgement}

\bibliographystyle{spmpsci}
\bibliography{references}  

\end{document}